\documentclass[12pt]{amsart}

\usepackage{amssymb,color}
\usepackage{amsfonts}
\usepackage{amsmath}
\usepackage{enumerate}

\begingroup

\newtheorem{thm}{Theorem\hskip 5mm}[section]

\newtheorem{prop}[thm]{Proposition\hskip 5mm}

\newtheorem{cor}[thm]{Corollary\hskip 5mm}

\newtheorem{lem}[thm]{Lemma\hskip 5mm}

\newtheorem{lemma}[thm]{Lemma\hskip 5mm}

\newtheorem{note}[thm]{Note\hskip 5mm}

\endgroup

\def\s{{\mathfrak s}}
\def\r{{\mathfrak r}}

\def\Sp{{\mathrm {Sp}}}

\def\GL{{\mathrm{ GL}}}

\def\m{{\mathfrak m}}
\def\i{{\mathfrak i}}
\def\j{{\mathfrak j}}
\def\r{{\mathfrak r}}

\def\s{{\mathfrak s}}
\def\k{{\mathfrak k}}

\begin{document}

\title[Unitary groups over local rings]
{Unitary groups over local rings}

\author{J. Cruickshank}
\address{School of Mathematics, Statistics and Applied Mathematics , National University of Ireland, Galway, Ireland}
\email{james.cruickshank@nuigalway.ie}

\author{A. Herman}
\address{Department of Mathematics and Statistics, Univeristy of Regina, Canada}
\email{Allen.Herman@uregina.ca}

\author{R. Quinlan}
\address{School of Mathematics, Statistics and Applied Mathematics , National University of Ireland, Galway, Ireland}
\email{rachel.quinlan@nuigalway.ie}

\author{F. Szechtman}
\address{Department of Mathematics and Statistics, Univeristy of Regina, Canada}
\email{fernando.szechtman@gmail.com}
\thanks{The second and fourth authors were supported in part by NSERC discovery grants}

\maketitle

\section{Introduction}

Let $K$ be a non-archimidean local field with ring of integers
$\mathcal{O}$, maximal ideal $\mathfrak p$ and residue field
$F_q=\mathcal{O}/\mathfrak p$ of odd characteristic, and let
$F$ be a ramified quadratic extension of $K$, with ring of integers
$\mathcal{R}$ and maximal ideal~$\mathcal{P}$. For the reader who is not acquainted
with this terminology, the following elementary properties will suffice for our purposes: $\mathcal{O}$ is a principal ideal domain
with a unique prime element $\pi$, up to multiplication by a unit, so that the non-zero ideals of $\mathcal{O}$ are
$$
\mathcal{O}=(\pi^0)\supset \mathfrak{p}=(\pi)\supset (\pi^2)\supset (\pi^3)\supset\cdots
$$
Moreover, $\mathcal{R}$ is a free $\mathcal{O}$-module of rank 2 obtained by adjoining to $\mathcal{O}$
a square root $\rho$ of $\pi$, i.e., $\mathcal{R}=\mathcal{O}[\rho]\cong \mathcal{O}[t]/(t^2-\pi)$. Here $\mathcal{R}$ is also
a principal ideal domain with a unique prime element up to units, namely $\rho$, so that $\mathcal{P}=(\rho)$, where
$\mathcal{R}/\mathcal{P}\cong F_q$. Furthermore, $K$ and $F$ are
the field of fractions of $\mathcal{O}$ and $\mathcal{R}$, respectively. There is an involution of $F$ that fixes $K$ elementwise,
preserves $\mathcal{R}$, and is determined by $\rho\mapsto -\rho$.

Given $\ell\geq 1$ we let $A=\mathcal{R}/\mathcal{P}^{2\ell}$, which is a finite, commutative, local and principal ring. Note that $A$ inherits an involution
from $\mathcal{R}$, whose fixed ring is $R=\mathcal{O}/\mathfrak{p}^\ell$. In \cite{CMS} the authors construct and decompose the Weil representation
$W:\Sp_{2m}(R)\to\GL(X)$ of the symplectic group. This is a complex representation of degree $|R|^m=q^{\ell m}$,
which generalizes the classical Weil representation of the symplectic group $\Sp_{2m}(q)$, as considered in \cite{PG}.

A unitary group $U_m(A)$ naturally imbeds in $\Sp_{2m}(R)$ and one
may consider the decomposition problem for the restriction of the
Weil representation to $U_m(A)$. This problem is solved in
\cite{HS}. As in the symplectic case, it is convenient to first
split the Weil module into its top and bottom layers. The latter
affords a Weil representation for a unitary group
$U_m(\widetilde{A})$, where $\widetilde{A}$ is a quotient ring of $A$,
so it suffices to achieve a decomposition for the top layer.
Unlike the symplectic case, the irreducible constituents of the
Weil module depend on the equivalence type of the underlying
non-degenerate hermitian form defining $U_m(A)$. There are two
types of such forms and this dependence occurs even when $m$ is
odd, when there is only one isomorphism type of $U_m(A)$.

What concerns us here is the degrees of the irreducible
constituents of the top layer of the Weil module for $U_m(A)$.
According to \cite{HS}, these degrees are equal to
$[U_m(B):S_u]/c$, where $c=2q^{\ell/2}$ if $\ell$ is even,
$c=2q^{(\ell-1)/2}$ if $\ell$ is odd, $B=A/\r^\ell$, $\r$ is the
maximal ideal of $A$, and $S_u$ is the stabilizer in $U_m(B)$ of a
primitive vector $u$, i.e., one belonging to a basis of the free
$B$-module of rank $m$ on which $U_m(B)$ operates. We wish to
compute the orders of the unitary group $U_m(B)$ and the
stabilizer~$S_u$ to determine the above degrees.

The purpose of this paper is twofold. We first develop from scratch a theory of unitary groups over local rings
that allows us to compute the order of finite unitary groups. We do not restrict ourselves to the above set-up,
but work throughout the paper with a general local ring with involution $A$, not necessarily commutative or principal, subject to certain technical
conditions, as described in \S\ref{bher} and \S\ref{sec:class}. Our involution has order 1 or 2, so orthogonal
groups are included as a special case. Our examples of primary interest are when $A=\mathcal{R}/\mathcal{P}^e$,
with involution inherited from $\mathcal{R}$, where $e\geq 1$ and $F/K$ is ramified or unramified, the latter meaning that $\mathcal{R}=\mathcal{O}[\tau]$, where $\tau^2$ is a unit in $\mathcal{O}$ that is not a square, as well as when $A=\mathcal{O}/\mathfrak{p}^e$, with identity involution.

The literature on hermitian forms and unitary groups over rings is quite extensive.
See for instance \cite{MK}, \cite{HO}, \cite{DJ} and references therein.
In spite of all that is known about this subject, we have not been able to find a specific reference
where the orders of these groups are explicitly displayed, except, of course, in the field case, as found in \cite{LG}. Thus, our formulas for these
orders, given in \S\ref{sec:Epimorphism},
may be of interest to the reader. A similar but incorrect formula appears in \cite{HF} in the case of unitary groups
of even rank. For orthogonal groups our formula agrees with a previous one given in \cite{YG} for commutative rings.
Our main tool to compute the desired orders is a canonical homomorphism $U_m(A)\to U_m(\overline{A})$, where $\overline{A}$
is a factor ring of~$A$. The fact that this is an epimorphism can be found in \cite{RB}, and we furnish an
independent and conceptually simple proof of this useful result. We also have a formula, given
in~\S\ref{sec:stapri}, for the order of the stabilizer $S_u$ in $U_m(A)$ of a primitive vector $u$ of arbitrary length. Perhaps
surprisingly, for primitive vectors of non-unit length, $|S_u|$ is independent of the length of $u$.

Our second focus of attention is the computation of the Weil character degrees, which can be achieved by means of the aforementioned
formulas for the orders of $U_m(A)$ and $S_u$. Care is required, as these character degrees are very sensitive to the type of hermitian form defining $U_m(A)$. The easiest case of these
computations, namely when $\ell=1$, was performed in \cite{HS} as it only relies on the orders of the orthogonal
groups over~$F_q$, already available in \cite{NJ}.

We are very grateful to R. Guralnick for his help interpreting the kernel of $U_m(A)\to U_m(\overline{A})$ when
the kernel of $A\to \overline{A}$ has square 0.

\section{Background on hermitian geometry}
\label{bher}

Let $A$ be a ring with identity, not necessarily commutative, with
Jacobson radical $\r$. We will assume that $A$ is local, in the
sense that $A/\r$ is a division ring. Thus, $\r$ is the unique
maximal left and right ideal of $A$ and every element of $A$ not
in $\r$ is in the unit group $A^*$ of $A$. Let $*$ be an
involution of $A$, i.e. an anti-automorphism of $A$ of order $\leq
2$. We will also assume that the elements fixed by $*$ are in the
center of~$A$. Thus they form a ring $R$, which is easily seen to
be local with Jacobson radical $\m=R\cap \r$ and residue field
$R/\m$. The norm map $Q:A^*\to R^*$, given by $a\mapsto aa^*$, is
a group homomorphism whose kernel will be denoted by $N$.

Let $V$ be a right $A$-module and let $h:V\times V\to A$ be a $*$-hermitian form. This means that $h$ is linear in the second variable
and satisfies
$$
h(v,u)=h(u,v)^*,\quad u,v\in V.
$$
Note that $h(u,u)\in R\subseteq Z(A)$ for every $u\in V$.

We can make the dual space $V^*$ into a right $A$-module by means of:
\begin{equation}
\label{dere}
(\alpha a)(v)=a^* \alpha(v),\quad v\in V,a\in A,\alpha\in V^*.
\end{equation}
We have a map $V\to V^*$, associated to $h$, given by $u\mapsto
h(u,-)$; it is a homomorphism of right $A$-modules. We will assume
that $h$ is non-degenerate, in the sense that $V\to V^*$ is an
isomorphism. The subgroup of $\GL(V)$ preserving $h$ will be
denoted by $U$.

We will also assume the existence of an element $d\in A$, necessarily in $A^*$, such that $d+d^*=1$. This is automatic if $2\notin \r$, i.e.,
if the characteristic of $R/\m$ is not~2, in which case we can take $d=1/2$.

Furthermore, we will suppose that $V$ is a free $A$-module of rank
$m\geq 1$. This is well-defined, as can be seen by reducing modulo
$\r$.

To avoid unnecessary repetitions, we fix a basis
$\{v_1,\dots,v_m\}$ of $V$ throughout this section.

\begin{lem}
\label{exun} There is a vector $u\in V$ such that $h(u,u)\in R^*$.
\end{lem}

\noindent{\it Proof.} Suppose, if possible, that $h(u,u)\in \m$
for all $u\in V$. From
$$
h(u+v,u+v)-h(u,u)-h(v,v)=h(u,v)+h(u,v)^*
$$
we deduce that $h(u,v)+h(u,v)^*\in \m$ for all $u,v\in V$.
Let $\alpha\in V^*$ be the linear functional satisfying $\alpha(v_1)=d$, $\alpha(v_i)=0$ for $i>1$.
By the non-degeneracy of $h$ there is $u\in V$ such that
$h(u,-)=\alpha$. Thus
 $h(u,v_1)=d$, so $h(u,v_1)+h(u,v_1)^*=d+d^*=1$, a contradiction.\qed

\begin{lem}
\label{orz} $V$ has an orthogonal basis $u_1,\dots,u_m$. Any such basis satisfies ~$h(u_i,u_i)\in R^*$.
\end{lem}

\noindent{\it Proof.} The second statement follows from the first and the non-degeneracy of $h$. We prove
the first statement by induction on $m$. If $m=1$ there is nothing to do. Suppose $m>1$ and the result is true for $m-1$.
By Lemma \ref{exun} there is $u\in V$ such that $h(u,u)\in R^*$. We have $u=v_1a_1+\dots+v_m a_m$, where $a_i\in A$.
If all $a_i\in\r$ then $h(u,u)\in\m$, a contradiction. Thus, without loss of generality, we may assume that $a_1\in A^*$.
Then $u_1,v_2,\dots,v_m$ is a basis of $V$. Set
$$
u_i=v_i-u_1[h(u_1,v_i)/h(u_1,u_1)],\quad 1<i\leq m.
$$
Then $u_1,u_2\dots,u_m$ is a basis of $V$ satisfying $h(u_1,u_i)=0$ for $1<i\leq m$. Let $V_1=Au_1$ and $V_2=\mathrm{span}\{u_2,\dots,u_m\}$.
Since $V=V_1\perp V_2$, the restriction of $h$ to $V_2$ induces an isomorphism $V_2\to V_2^*$ and the inductive hypothesis applies.\qed

\begin{lem}
\label{orz2} (a) Suppose $u_1,\dots,u_s\in V$ are orthogonal and satisfy $h(u_i,u_i)\in R^*$. Then $u_1,\dots,u_s\in V$
can be extended to an orthogonal basis of $V$ with the same property.

(b) If $V_1$ is a free submodule of $V$ such that the restriction of $h$ to $V_1$ is non-degenerate
there is another such submodule $V_2$ of $V$ such that  $V=V_1\perp V_2$.
\end{lem}

\noindent{\it Proof.}  (a) We have $u_1=v_1a_1+\cdots+v_ma_m$ for some $a_i\in A$.
Since $h(u_1,u_1)\in R^*$ one of these scalars must be a unit, say $a_1$.
Thus $u_1,v_2,\dots,v_m$ is a basis of $V$. Suppose $1\leq t<s$ and we have shown that, up to a reordering of the $v_i$,
the list $u_1,\cdots,u_t,v_{t+1},\dots,v_m$ is a basis of $V$. We have $u_{t+1}=u_1b_1+\cdots+u_tb_t+v_{t+1}b_{t+1}+\cdots+v_mb_m$ for some $b_i\in A$.
Suppose, if possible, that $b_i\in\r$ for all $i\geq t+1$. Then
$$0=h(u_i,u_{t+1})=h(u_i,u_i)b_i+h(u_i,v_{t+1})b_{t+1}+\cdots+h(u_i,v_{m})b_{m},\quad 1\leq i\leq t,
$$
so $b_i\in\r$ for all $1\leq i\leq t$, which contradicts $h(u_{t+1},u_{t+1})\in R^*$. Thus at least one of $b_{t+1},\dots,b_m$ must be a unit,
say $b_{t+1}$, and $u_1,\cdots,u_t,u_{t+1},v_{t+2},\dots,v_m$ is a basis of $V$.

This shows that $u_1,\dots,u_s$ can be extended to a basis $u_1,\dots,u_s,u_{s+1},\dots,u_m$ of~$V$. Let
$$
z_i=u_i-\left([u_1h(u_1,u_i)/h(u_1,u_1)]+\dots+[u_sh(u_s,u_i)/h(u_s,u_s)]\right),\quad s<i\leq m.
$$
Then $u_1,\dots,u_s,z_{1},\dots,z_{m-s}$ is a basis of $V$ satisfying $h(u_i,z_j)=0$. It follows that the restriction
of $h$ to $M=\mathrm{span}\{z_1,\dots,z_{m-s}\}$ is non-degenerate, so by Lemma \ref{orz} $M$ has an orthogonal
basis with lengths in $R^*$.

(b) This follows from (a) and Lemma \ref{orz}.\qed

\begin{lem}
\label{nueva} Let $u_1,\dots,u_s\in V$, with corresponding Gram
matrix $M\in M_s(A)$, defined by $M_{ij}=h(u_i,u_j)$. If $M\in
\GL_m(A)$ then $u_1,\dots,u_s$ are linearly independent.
\end{lem}

\noindent{\it Proof.} Suppose $a_1,\dots,a_s$ satisfy $u_1a_1+\dots+u_sa_s=0$. Then
$$
0=h(u_i,u_1 a_1+\dots+u_s a_s)=h(u_i,u_1)a_1+\cdots+h(u_i,u_s)a_s,\quad 1\leq i\leq s.
$$
This means
$$
M\left(
   \begin{array}{c}
     a_1 \\
     \vdots \\
     a_s \\
   \end{array}
 \right)=\left(
   \begin{array}{c}
     0 \\
     \vdots \\
     0 \\
   \end{array}
 \right).
 $$
Since $M$ is invertible, the result follows.\qed

\section{Classification of hermitian forms}
\label{sec:class}

A vector $v\in V$ is said to be primitive if $v\not\in V\r$. This is equivalent to saying that $v$ belongs to a basis of $V$.
We say that $h$ is isotropic if there is a primitive vector $v\in V$ such that $h(v,v)=0$.

\begin{lem}
\label{hiso} Suppose $h$ is isotropic. Then, given any $r\in R$ there is a primitive vector $v\in V$ satisfying $h(v,v)=r$.
\end{lem}

\noindent{\it Proof.} By assumption $h(u,u)=0$ for some $u\in V$
primitive. By the non-degeneracy of $h$ there is $w\in V$ such
that $h(u,w)=d$. Set $s=r-h(w,w)\in R$ and $v=us+w$. Then
$$
h(v,v)=h(us+w,us+w)=sh(u,w)+sh(w,u)+h(w,w)=s+h(w,w)=r-h(w,w)+h(w,w)=r.
$$
Since $h(u,v)=h(u,w)=d\in A^*$, it follows that $v$ is primitive.\qed

\medskip

We assume for the remainder of the paper that the squaring map of the 1-group $1+\m$ of $R$ is an epimorphism and that $R/\m=F_q$
is a finite field of order $q$ and odd characteristic. Thus $[F_q^*:F_q^{*2}]=2$. It follows that $[R^*:R^{*2}]=2$,
and we fix an element $\epsilon$ of $R^*$ not in $R^{*2}$. Since $R^{*2}\subseteq Q(A^*)$, we infer
$Q(A^*)=R^*$ if $Q$ is surjective and $Q(A^*)=R^{*2}$ otherwise.

\begin{prop}
\label{45} The division ring $A/\r$ is commutative. Moreover,

(a) If the involution that $*$ induces on $A/\r$ is the identity then $Q$ is not surjective and $A/\r\cong F_q$.

(b) If the involution that $*$ induces on $A/\r$ is not the identity then $Q$ is surjective and $A/\r\cong F_{q^2}$.
\end{prop}

\noindent{\it Proof.} We have an imbedding $R/\m\to A/\r$, given by $r+\m\mapsto r+\r$, so we may view
$R/\m$ as  a subfield of $A/\r$. Let $\circ$ be the involution that $*$ induces on $A/\r$ and
let $k$ consist of all elements of $A/\r$ that are fixed by $\circ$. Clearly $k$ contains $R/\m$. Conversely, if $a+\r\in k$ then $a-a^*\in\r$ so
$a=(a+a^*)/2+(a-a^*)/2$ is in $R+\r$. Thus $k=(R+\r)/\r$, i.e., $k=R/\m$.

(a) In this case $A/\r=k=R/\m$. Thus the induced norm
map $(A/\r)^*\to (R/\m)^*$ is the squaring map of $F_q^*$, which is not surjective, so the norm map $A^*\to R^*$
is not surjective.

(b) By assumption $D=A/\r$ properly contains $k$. Note that $D$ is algebraic over $k$. In fact, the minimal polynomial over $k$ of every element
$f\in D\setminus k$ is $(t-f)(t-f^\circ)=t^2-(f+f^\circ)t+ff^\circ\in k[t]$.

Let $f,e\in D$. We wish to see that $f,e$ commute. Let $f_1=f-(f+f^\circ)/2$ and $e_1=e-(e+e^\circ)/2$. Since $(f+f^\circ)/2,(e+e^\circ)/2\in k$,
it suffices to show that $f_1,e_1$ commute. Now $f_1^\circ=-f_1$ and $e_1^\circ=-e_1$, so $e_1f_1+f_1e_1\in k$.
It follows that $k\langle f_1,e_1\rangle$ is the $k$-span of $1,f_1,e_1,f_1e_1$. Thus $k\langle f_1,e_1\rangle$ is a finite dimensional division algebra over $k$, hence a field by Wedderburn's Theorem.

Thus $D$ is a field, algebraic over $k$, and every element of $D\setminus k$ has degree 2 over~$k$. Since every algebraic extension of $k$ is separable, the theorem of the primitive element ensures that $[D:k]=2$, so $D\cong F_{q^2}$. It follows that the norm map $(A/\r)^*\to (R/\m)^*$ induced $*$ is surjective, and hence so is
the norm map $A^*\to R^*$ since the squaring map of $1+\m$ is surjective.\qed

\begin{note}
\label{noncom} {\rm $A$ itself need not be commutative. Indeed,
let $q$ be a power of an odd prime and let $S$ be the skew
polynomial ring in one variable $t$ over $F_{q^2}$, where $ta=a^q
t$, $a\in F_{q^2}$. We have an involution on $S$ given by
$a_0+a_1t+a_2t^2\cdots\mapsto a_0^q-a_1t+a_2^q t^2+\cdots$. Let
$n\geq 2$ and set $A=S/(t^n)$, which inherits an involution $*$
from~$S$.
 The set $R$ of elements fixed by $*$ is $b_0+b_2t^2+b_4t^4+\cdots$, where $b_i\in F_q$,
which is central in~$A$. Moreover, $A$ is a non-commutative principal ring, with $A/\r\cong F_{q^2}$, $R/\m\cong F_q$.
Since $|\m|$ is odd, the squaring map of $1+\m$ is an automorphism.}
\end{note}

\begin{prop}
\label{casom2} Suppose $m\geq 2$.
Then given any unit $r\in R$ there is a primitive vector $v\in V$ satisfying $h(v,v)=r$.
\end{prop}

\noindent{\it Proof.} We consider two cases:

\noindent$\bullet$ $h$ is isotropic. Then Lemma \ref{hiso} applies.

\noindent$\bullet$ $h$ is non-isotropic. By Lemma \ref{orz} there
is an orthogonal basis $u_1,u_2,\dots,u_m$ of $V$ such that
$h(u_i,u_i)\in R^*$. Let $a=h(u_1,u_1)\in R^*$ and
$b=h(u_2,u_2)\in R^*$. If $t_1,t_2\in R^*$ then $v=u_1t_1+u_2t_2$
is primitive, so
$$
0\neq h(v,v)=at_1^2+bt_2^2.
$$
Dividing by $a$ and letting $c=b/a\in R^*$ we have
$$
0\neq t_1^2+ct_2^2,\quad t_1,t_2\in R^*.
$$
In particular, $-c$ is not a square in $R^*$. Set $S=R[t]/(t^2+c)$ and $\delta=t+(t^2+c)\in S$. Then $S=R[\delta]$, $\delta^2=-c$,
and every element of $S$ can be uniquely written in the form $t_1+t_2\delta$ with $t_1,t_2\in R$. We have an involution $s\mapsto\hat{s}$ of $S$
defined by $t_1+t_2\delta=t_1-t_2\delta$, whose corresponding norm map $J:S^*\to R^*$ is given by $s\mapsto s\hat{s}$, that is,
$t_1+t_2\delta\mapsto t_1^2+ct_2^2$.

We claim that $S$ is local with maximal ideal $S\m$. Indeed, let
$t_1,t_2\in R$, not both in $\m$, and consider
$J(t_1+t_2\delta)=t_1^2+ct_2^2$. If one of $t_1,t_2$ is in $\m$
then we see that $t_1^2+ct_2^2\in R^*$ so $t_1+t_2\delta\in S^*$.
Suppose, if possible, that $t_1,t_2\in R^*$ but
$t_1+t_2\delta\notin S^*$. Then $t_1^2+ct_2^2=f\in\m$, so
$$-c=(t_2^{-1})^2(t_1^2-f)=(t_2^{-1})^2t_1^2(1-(t_1^{-1})^2f).
$$
By assumption $1-(t_1^{-1})^2f\in R^{*2}$, so $-c\in R^{*2}$, a contradiction.

Thus $S/S\m$ is a field. We have an imbedding $R/\m\to S/S\m$ so may view $S/S\m$ as a vector space over $R/\m$.
Clearly $\{1+S\m,\delta+S\m\}$ is a basis, so  $S/S\m$ is a quadratic extension of $R/\m$. The involution of $S$
induces the $R/\m$-automorphism of $S/S\m$ of order 2 and the norm map $J$ induces the norm map $(S/S\m)^*\to (R/\m)^*$.
Since $R/\m=F_q$, this map is known to be surjective. We claim that $J$ is surjective. Indeed, let $e\in R^*$. Then there is $s\in S$ and $f\in\m$
such that $$J(s)=e+f=e(1+e^{-1}f).$$
Since $1+e^{-1}f\in R^{*2}$, it follows that $e$ is in the image of $J$, as claimed.

By the claim there are $t_1,t_2\in R$, at least one of them a unit, such that $t_1^2+t_2^2c=r/a$. Thus $v=u_1t_1+u_2t_2$ is primitive and we have
$$
h(v,v)=at_1^2+bt_2^2=r.\qed
$$

\begin{thm}
\label{ort} There is an orthogonal basis $v_1,\dots,v_m$ of $V$ satisfying
$$h(v_1,v_1)=\dots=h(v_{m-1},v_{m-1})=1\text{ and }$$
$$h(v_m,v_m)=1\text{ if }Q(A^*)=R^*,\quad h(v_m,v_m)\in\{1,\epsilon\}\text{ if }Q(A^*)=R^{*2}.$$
\end{thm}

\noindent{\it Proof.} We show the existence of such a basis by
induction on $m$. The result is clear if $m=1$. Suppose $m>1$ and
the result is true for $m-1$. By Proposition \ref{casom2} there is
a primitive vector $u_1\in V$ such that $h(u_1,u_1)=1$. By Lemma
\ref{orz2} there is an orthogonal basis $u_1,u_2,\dots,u_m$ of $V$
such that $h(u_i,u_i)\in R^*$, and the inductive hypothesis
applies.\qed

\medskip

Let $\i$ be a $*$-invariant ideal of $A$ and let $\overline{A}=A/\i$. Then $*$ induces an involution on $\overline{A}$. Moreover,
$\overline{V}=V/\i V$ is a free  $\overline{A}$-module of rank $m$ and the map $\overline{h}: \overline{V}\times  \overline{V}\to \overline{A}$,
given by $\overline{h}(v+\i V,w+\i V)=h(v,w)$, is a non-degenerate hermitian form.

Recall that when $A$ is commutative the discriminant of $h$ is the element of $R^*/Q(A^*)$ obtained by taking the determinant
of the Gram matrix of $h$ relative to any basis of $V$.

\begin{cor}
\label{sant} Let $h_1,h_2$ be non-degenerate hermitian forms on $V$. Then the following conditions are equivalent:

(a) $h_1$ and $h_2$ are equivalent.

(b) The reductions $\overline{h_1}$ and $\overline{h_2}$ of $h_1$ and $h_2$ modulo $\r$ are equivalent.

(c) The discriminants of $\overline{h_1}$ and $\overline{h_2}$ are the same.

\noindent In particular, all non-degenerate hermitian forms are
equivalent when $Q$ is surjective and there are exactly two such
forms, up to equivalence, when $Q$ is not surjective. Moreover, if
$A$ is commutative then the conditions (a)-(c) are equivalent to:

(d) The discriminants of $h_1$ and $h_2$ are the same.
\end{cor}

Given $r_1,\dots,r_m\in R^*$ we say that $h$ is of type $\{r_1,\dots,r_m\}$ if there is a basis $B$ of $V$ relative to which $h$ has
matrix $\mathrm{diag}\{r_1,\dots,r_m\}$. In that case, $h$ is also of type  $\{s_1,\dots,s_m\}$, for $s_i\in R^*$, if and only if
$(r_1\cdots r_m)(s_1\cdots s_m)^{-1}\in Q(A^*)$.

When $m$ is even then $h$ is of type $\{1,-1,\dots,1,-1\}$ (kind I) or
$\{1,-1,\dots,1,-\varepsilon\}$ (kind~II).
When $m$ is odd then $h$ is of type $\{1,-1,\dots,1,-1,-1\}$ (kind $I$) or
$\{1,-1,\dots,1,-1,-\varepsilon\}$ (kind~II). Note that $h$ is of kind I and II if and only if the norm map $Q:A^*\to R^*$
is surjective.

Even when $Q:A^*\to R^*$ is not surjective, if $m$ is odd there is only one unitary group of rank~$m$, regardless of $h$, since $h$ and $\epsilon h$ are non-equivalent and have the same unitary group.

\begin{lem}
\label{tiso} Let $\Lambda$ be the set of all values $h(u,u)$ with
$u\in V$ primitive. Assume that the involution that $*$ induces on
$A/\r$ is the identity.

(a) Suppose $m=1$. If $h$ is of type $\{1\}$ then $\Lambda=R^{*2}$
and if $h$ is of type $\{\epsilon\}$ then $\Lambda=R^*\setminus
R^{*2}$.

(b) Suppose $m=2$. If $h$ is of type $\{1,-1\}$ then $\Lambda=R$ and
if $h$ is of type $\{1,-\epsilon\}$ then $\Lambda=R^*$.

(c) If $m>2$ then  $\Lambda=R$.
\end{lem}

\noindent{\it Proof.} (a) This is obvious.

(b) If $h$ is of type $\{1,-1\}$ then $h$ is isotropic and Lemma \ref{hiso} applies. Suppose $h$ is of type $\{1,-\epsilon\}$.  If possible, let $v=v_1r+v_2s$ be primitive and satisfy $h(v,v)\in \m$.
One of the coefficients, say $r$, is a unit. Since $rr^*-\epsilon ss^*=f$,
not a unit, it follows that $s$ is also a unit.
Both $rr^*$ and $ss^*$ are squares in $R^*$, so we have
$t_1^2-\epsilon t_2^2=f$ for some $t_1,t_2\in R^*$. Reducing modulo $\m$ we
get $a^2-\delta b^2=0$ in $F_q$, with $a, b,\delta\neq 0$ and $\delta\notin F^{*2}_q$, which is impossible.
This shows that no values from $\m$ are attained. On the other hand, by Corollary \ref{sant} type
$\{1,-\epsilon\}$ is equivalent to type $\{-1,\epsilon\}$, so both squares and non-squares units are attained.

(c) Let $v_1,v_2,v_3,\dots,v_m$ be an orthogonal basis of $V$.
Since $-h(v_3,v_3)\in R^*$, Proposition \ref{casom2} ensures the
existence of a primitive vector $v$ of $v_1A\oplus v_2A$ such that
$h(v,v)=-h(v_3,v_3)$. Then $u=v+v_3$ is primitive and
$h(u,u)=0$.\qed

\section{Primitive vectors of the same length are $U$-conjugate}

The following is essentially a weak form of Witt's extension theorem.

\begin{thm}\label{thm:transitivity}
    Let $v,w\in V$ be primitive vectors satisfying $h(v,v)=h(w,w)$. Then there exists $g\in U$ such that $gv=w$.
\label{crux}
\end{thm}

\noindent{\it Proof.} The proof is easy if $m=1$, so we assume $m>1$.

Let $h(v,v)=r=h(w,w)\in R$. Two cases arise.

\noindent$\bullet$ $r$ is a unit. By Lemma \ref{orz2} we have
$V=vA\perp V_1$, where $V_1$ is free of rank $m-1$ and the
restriction of $h$ to $V_1$ is non-degenerate. We have a similar
decomposition $V=wA\perp V_2$. Since $h(v,v)=r=h(w,w)$, we may
apply Corollary \ref{sant} to see that the restrictions of $h$ to
$V_1$ and $V_2$ are equivalent. Thus, there is a linear
isomorphism $f:V_1\to V_2$ that preserves $h$ on both spaces. We
can extend this to a linear isomorphism $V\to V$ that preserves
$h$ and sends $v$ to $w$, as required.

\noindent$\bullet$ $r$ is not a unit. Since $v$ is primitive there is a vector $u\in V$ such that $h(v,u)=1$. The Gram matrix corresponding
to $v,u$ is
$$
M=\left(
  \begin{array}{cc}
    r & 1 \\
    1 & s \\
  \end{array}\right),\text{ where }s=h(u,u).
$$
Since the determinant of $M$ is the unit $-1+rs$, it follows from
Lemma \ref{nueva} that $v,u$ are linearly independent. Now, the
determinant of
$$
N=\left(
  \begin{array}{cc}
    0 & 1 \\
    1 & 0 \\
  \end{array}
\right)
$$
is $-1$. Thus $M$ and $N$ are congruent by Corollary \ref{sant}.
Thus $vA\oplus uA$ has a basis $z_1,z_2$ whose Gram matrix is $N$.
We have $v=z_1a_1+z_2a_2$ with $a_1,a_2\in A$. Since $v$ is
primitive, one of this coefficients is a unit, say $a_1$. Let
$z=(a_1^{-1})^*z_2$. Then $v,z$ are linearly independent with Gram
matrix
$$
C=\left(
  \begin{array}{cc}
    r & 1 \\
    1 & 0 \\
  \end{array}
\right).
$$
By Lemma \ref{orz2} we have $V=(vA\oplus zA)\perp V_1$, where $V_1$ is a free $A$-module of rank $m-2$
and the restriction of $h$ to $V_1$ is non-degenerate.

Likewise there is a vector $z'\in V$ such that $w,z'$ are linearly
independent with Gram matrix $C$, and an analogous decomposition $V=(wA\oplus z'A)\perp V_2$. Since the Gram matrices corresponding to
$v,z$ and $w,z'$ are the same, the restrictions of $h$ to $V_1$ and $V_2$ are equivalent by Corollary \ref{sant},
and the same argument used in the first case applies.\qed

\section{A canonical epimorphism}
\label{sec:Epimorphism}

\begin{prop}
\label{33} Let $\i$ be a $*$-invariant nilpotent ideal of $A$.
Let $u_1,\dots,u_m$ be an orthogonal basis of $V$ and let $v_1,\dots,v_m\in V$.
Suppose
$$
h(v_i,v_j)\equiv h(u_i,u_j)\mod \r,\quad 1\leq i,j\leq m
$$
and
$$
h(v_1,v_i)\equiv h(u_1,u_i)\mod \i,\quad 1\leq i\leq m.
$$
Then there is $w\in \i V$ such that
$$
h(v_1+w,v_i+w)=h(u_1,u_i),\quad 1\leq i\leq m.
$$
\end{prop}

\noindent{\it Proof.}  By induction on the nilpotency degree of $\i$. If $\i=(0)$ we may just take $w=0$.
Suppose $\i\neq (0)$ is nilpotent and the result is true for nilpotent ideals of nilpotency degree smaller than that of $\i$.

Let $M$ be the Gram matrix of $h$ relative to $v_1,\dots,v_m$. Our first assumption implies that the reduction of $M$ modulo $\r$
is invertible. Since $M_m(\r)$ is the Jacobson radical of $M_m(A)$, we see that $M$ must be invertible. We deduce from Lemmas \ref{orz2} and \ref{nueva} that  $v_1,\dots,v_m$ is a basis of~$V$.

By hypothesis we have $h(u_1,u_i)=h(v_1,v_i)+a_i$, where $a_i\in\i$ and $a_1\in\i\cap R$.
Since $h$ is non-degenerate and $v_1,\dots,v_m$ is a basis of $V$ there is $w\in V$ such that
$$
h(w,v_1)=a_1/2\text{ and  }h(w,v_i)=a_i-a_1/2\text{ if }1<i\leq m.
$$
Since $a_1,\dots,a_m\in\i$, we see that $h(w,V)\subseteq \i$. By
means of the orthogonal basis $u_1,\dots,u_m$ we see that $w\in
V\i$. Moreover, a direct calculation, using the choice of~$w$,
shows that
$$
h(v_1+w,v_i+w)\equiv h(u_1,u_i)\mod \i^*\,\i,\quad 1\leq i\leq m.
$$
Since $\i$ is $*$-invariant, these congruences hold modulo $\i^2$, which is a nilpotent ideal with nilpotency degree less than
that of $\i$. Moreover, since $w\in  V\i\subseteq V\r $, we have
$$
h(v_i+w,v_j+w)\equiv h(v_i,v_j)\equiv h(u_i,u_j)\mod \r.
$$
The inductive hypothesis ensures the existence of a vector $w'\in V\i^2 $ such that
$$
h(v_1+w+w',v_i+w+w')= h(u_1,u_i),\quad 1\leq i\leq m.
$$
Since $\i^2$ is included in $\i$, the vector $w+w'$ is in $V\i$ and satisfies the stated conditions.\qed

\medskip

Let $\i$ be a $*$-invariant ideal of $A$ and let $\overline{A},\overline{V}$ and $\overline{h}$ be defined as in \S\ref{sec:class}.
Note that $\overline{h}$ is of the same kind as $h$, in the sense defined in \S\ref{sec:class}.
Let $\overline{U}$ be the unitary group associated to $(\overline{V},\overline{h})$.
We have a canonical group epimorphism $U\to \overline{U}$, given by $g\mapsto \overline{g}$, where $\overline{g}(v+V\i)=g(v)+V\i$.

\begin{thm}
\label{tow}
 Suppose $\i$ is a $*$-invariant nilpotent ideal of $A$. Then the canonical homomorphism $U\to \overline{U}$ is surjective.
\end{thm}

\noindent{\it Proof.} By induction on $m$. Suppose first $m=1$. In this case the unitary group $\overline{U}$ coincides with the 1-norm
group of $\overline{A}$, so it consists of all $z+\i$ such that $z\in A$ and $zz^*\equiv 1\mod \i$. Now
$zz^*=1+s$ with $s\in\i\cap R$. Since $\i\cap R$ is nilpotent there is $r\in\i\cap R$ such that $1+s=(1+r)^2=(1+r)(1+r)^*$.
Now $z(1+r)^{-1}$ has norm 1, so $z=z_1(1+r)$ for some $z_1\in N$. Since $N=U$ and $z+\i=z_1(1+r)+\i=z_1+\i$, this case is established.

Suppose $m>1$ and the result is true for $m-1$. Let $f\in \overline{U}$. We wish to find $g\in U$ such that $\overline{g}=f$.
For this purpose, let $u_1,\dots,u_m$ be an orthogonal basis of $V$. Then $f(u_i+V\i)=v_i+V\i$ for some $v_1,\dots,v_m\in V$
satisfying $h(v_i,v_j)\equiv h(u_i,u_j)\mod \i$. By Proposition \ref{33} there is $w\in V$ such that
$$
h(v_1+w,v_1+w)=h(u_1,u_1)\text{ and }h(v_1+w,v_i+w)=0\text{ if }1<i\leq m.
$$
By Theorem \ref{crux} there is $k\in U$ such that $ku_1=v_1+w$.
Note that $k$ maps $V_0=u_1^\perp=\langle u_2,\dots,u_m\rangle$
into $V_1=(v_1+w)^\perp=\langle v_2+w,\dots,v_m+w\rangle$ and its
restriction, say $k_0$, is an equivalence between these
non-degenerate hermitian spaces. Moreover, we have the equivalence
$f_0:\overline{V_0}\to \overline{V_1}$, given by
$f_0(u_i+V\i)=v_i+V\i$. Therefore $\overline{k_0}^{-1}f_0$ is in
the unitary group of $\overline{V_0}$. By the inductive hypothesis
there is $g_1\in U'$, the unitary group of $V_0$, such that
$\overline{k_0}^{-1}f_0=\overline{g_1}$. Therefore
$f_0=\overline{k_0g_1}$, where $k_0g_1:V_0\to V_1$ is an
equivalence. If we now let $g=k\oplus k_0g_1$ then $g\in U$ and
$\overline{g}=f$.\qed

\begin{lem}
\label{cuad}
 Suppose $\i$ is a $*$-invariant ideal of $A$ satisfying $\i^2=0$. Let $\{v_1,\dots,v_m\}$ be a basis of $V$
and let $X$ be the Gram matrix of $h$ relative to
$\{v_1,\dots,v_m\}$. Then, relative to $\{v_1,\dots,v_m\}$, the
kernel of the canonical epimorphism $U\to \overline{U}$ consists
of all matrices $1+M$, such that $M\in M_m(\i)$ and
\begin{equation}
\label{anr}
{M^*}'X+XM=0,
\end{equation}
where ${M^*}'$ denotes the transpose of $M^*$.
\end{lem}

\noindent{\it Proof.} By definition the kernel of $U\to \overline{U}$ consists of all matrices of the form $1+M$,
where $M\in M_m(\i)$ and
$$
{(1+M)^*}'X(1+M)=X.
$$
Expanding this equation and using $\i^2=0$ yields (\ref{anr}).\qed

For the remainder of this paper we assume that $A$ is a finite ring. In particular, its maximal ideal $\r$ is nilpotent,
and we denote by $e\geq 1$ the nilpotency degree of $\r$.

\begin{cor}
\label{cuad2} The kernel of  $U\to \overline{U}$ has order
$\i^{m(m-1)/2}|\k|^m$, where $\k$ is the group of all $a\in\i$
such that $a+a^*=0$.
\end{cor}

\noindent{\it Proof.} We may choose $\{v_1,\dots,v_m\}$ so that $X$ is diagonal with unit entries. A direct calculation based on (\ref{anr}) now yields
the desired result.\qed

\begin{thm}
\label{zxz} Let $A$ be a local ring, not necessarily commutative, with Jacobson radical~$\r$. Suppose $A$ is finite and has
an involution $*$ whose set of fixed points, say $R$, lies in the center of $A$. Let $\m$ be the Jacobson
radical of the finite local commutative ring $R$ and suppose that the residue field $R/\m\cong F_q$ has odd characteristic.
Let $\s$ be the kernel of the trace map $\r\to\m$.

Let $V$ a free right $A$-module of rank $m\geq 1$ equipped with a non-degenerate $*$-hermitian form~$h$. Let $U_m(A)$ and $U_m(\overline{A})$
be the unitary groups associated to $(V,h)$ and its reduction $(\overline{V},\overline{h})$ modulo $\r$, as defined in \S\ref{sec:class}. Then
$$
|U_m(A)|=|\r|^{m(m-1)/2}|\s|^m |U_m(\overline{A})|=|\r|^{m(m+1)/2} |U_m(\overline{A})|/|\m|^m.
$$
\end{thm}

\noindent{\it Proof.} Let $e\geq 1$ be the nilpotency degree of $\r$ and consider the rings
$$
A=A/\r^{e}, A/\r^{e-1},\dots,A/\r^2,A/\r.
$$
Each of them is a factor of $A$, so is local and inherits an
involution from $*$. Each successive pair is of the form
$C=A/\r^k,D=A/\r^{k-1}$,  where the kernel of the canonical
epimorphism $C\to D$ is $\j=\r^{k-1}/\r^k$, so that $\j^2=0$. We
may thus apply Theorem \ref{tow} and Corollary \ref{cuad2} $e-1$
times to obtain the desired result, as follows. First of all,
$$
|\r|=|\r^{e-1}/\r^e|\cdots |\r/\r^2|.
$$
Secondly, $\s$ is the group of all $a\in\r$ satisfying $a+a^*=0$, with
$$
|\s|=|\s\cap\r^{e-1}/\s\cap \r^e|\cdots |\s\cap\r^{k-1}/\s\cap \r^{k}|\cdots |\s\cap \r/\s\cap \r^2|.
$$
Here the group of elements of trace 0 in the kernel of $C\to D$
has $|\s\cap\r^{k-1}/\s\cap \r^{k}|$ elements. Indeed, these
elements are those $a+\r^k$ such that $a\in \r^{k-1}$ and
$a+a^*\in \r^k$. But $a-a^*$ is an element of trace 0, so
$a-a^*\in\s\cap\r^{k-1}$. Thus $$a=(a-a^*)/2+(a+a^*)/2\in
\s\cap\r^{k-1}+\r^k.$$ Hence the group of elements of trace 0 in
the kernel of $C\to D$ is
$$
(\s\cap\r^{k-1}+\r^k)/\r^k\cong \s\cap\r^{k-1}/(\s\cap\r^{k-1}\cap \r^k)\cong \s\cap\r^{k-1}/\s\cap\r^{k}.
$$
Finally, since $2\in R^*$ the trace map $\r\to\m$ is surjective,
so its kernel is $\s\cong \r/\m$.\qed

\begin{cor}
\label{pom}
 Let $K,F,\mathcal{O},\mathcal{R},\mathfrak{p},\mathcal{P}, q$ be as in the Introduction and let $e\geq 1$.

(a) Let $A=\mathcal{O}/\mathfrak{p}^e$, with identity involution. Then
$$
|O_m(A)|=q^{m(m-1)(e-1)/2}|O_m(q)|,
$$
where $O_m(q)$ is the orthogonal group of rank $m$ over $F_q$ associated to the reduction $\overline{h}$ of $h$ modulo~$\r$.

(b) Let $A=\mathcal{R}/\mathcal{P}^e$ with involution inherited from $\mathcal{R}$. If $F/K$ is unramified then
$$
|U_m(A)|=q^{m^2(e-1)}|U_m(q^2)|.
$$
If $F/K$ is ramified and $e$ is even then
$$
|U_m(A)|=q^{(m^2(e-1)+m)/2}|O_m(q)|
$$
and if $F/K$ is ramified and $e$ is odd then
$$
|U_m(A)|=q^{m^2(e-1)/2}|O_m(q)|.
$$
Here $O_m(q)$ is the orthogonal group of rank $m$ over $F_q$ associated to the reduction $\overline{h}$ of $h$ modulo~$\r$.
\end{cor}

\noindent{\it Proof.} (a) Use $\r=\m$ and Theorem \ref{zxz}.

(b) In the unramified case $R\cong \mathcal{O}/\mathfrak{p}^e$ and $\mathcal{R}/\mathcal{P}\cong F_{q^2}$. Thus
$|\r|=q^{2(e-1)}$ and $|\m|=q^{e-1}$. In the ramified case $\mathcal{R}/\mathcal{P}\cong F_{q}$ so $|\r|=q^{e-1}$. Moreover, if $e$ even then $R\cong \mathcal{O}/\mathfrak{p}^{e/2}$
so that $|\m|=q^{(e-2)/2}$, and if $e$ is odd then $R\cong \mathcal{O}/\mathfrak{p}^{(e+1)/2}$ so that $|\m|=q^{(e-1)/2}$. Now apply Theorem ~\ref{zxz}.\qed

\smallskip

For the purpose of calculating the character degrees mentioned in the Introduction, we isolate the following case.

\begin{cor}
\label{pom2} Let $K,F,\mathcal{O},\mathcal{R},\mathfrak{p},\mathcal{P}, q$ be as in the Introduction and let $\ell\geq 1$.
Suppose $F/K$ is ramified. Set $A=\mathcal{R}/\mathcal{P}^{2\ell}$ and $B=A/\r^\ell$, both with involutions inherited from $\mathcal{R}$.
If $\ell$ is even then
$$
|U_m(B)|=q^{(m^2(\ell-1)+m)/2}|O_m(q)|
$$
and if $\ell$ is odd then
$$
|U_m(B)|=q^{m^2(\ell-1)/2}|O_m(q)|.
$$
Here $O_m(q)$ is the orthogonal group of rank $m$ over $F_q$ associated to the reduction $\overline{h}$ of $h$ modulo~$\r$.
\end{cor}

\noindent{\it Proof.} Since $B\cong \mathcal{R}/\mathcal{P}^{\ell}$ as a ring with involution, we may apply Corollary \ref{pom} with $e=\ell$.\qed

\section{Unitary groups of rank 2: an alternative approach}
\label{sec:rank2}

Here we compute the order of $U_2(A)$ in a different and direct way, without resorting to the field case.

Given $s\in R$, let $d(s)$ be the number
of primitive vectors $v\in V$ such that
$$
h(v,v)=s.
$$

Suppose $u\in V$ is a primitive of length $h(u,u)=s$.
Then, by Theorem \ref{crux}, $d(s)$ is the number of $U$-conjugates to $v\in V$, i.e., $d(s)$
is the index of the stabilizer $S_u$ of $u$ in $U$. Note that when $h(u,u)\in R^*$ then $S_u$ is the unitary
group of a hermitian space of rank $m-1$.

By means of an orthogonal basis $v_1,\dots,v_m$  we see that
$d(s)$ is the number of solutions to
$$
\delta_1x_1^*x_1+\cdots+\delta_m x_{m}^*x_{m}=s,
$$
where $\delta_i=h(v_i,v_i)\in R^*$, all $x_i\in A$, and at least one $x_i$ is a unit.

\begin{lem}
\label{nuni} Let $m=2$ and suppose $s\in\m$. Then $d(s)=0$ if $h$ is not isotropic and
$$
d(s)=|A^*|\times |N|.
$$
if $h$ is isotropic.
\end{lem}

\noindent{\it Proof.} If $h$ is non-isotropic then $d(s)=0$ by Lemma \ref{tiso}. Suppose $h$ is isotropic. Then $h$ is of kind I.
If $x_1^*x_1-x_2^*x_2=s$, with $x_1$ or $x_2$ in $A^*$, then both $x_1,x_2$ are in $A^*$. Here
$x_2$ can be chosen arbitrarily in $A^*$ and, since $s+x_2^*x_2\in R^{*2}$, there are $|N|$ choices for
$x_1\in A^*$ to satisfy $x_1^*x_1=s+x_2^*x_2$.\qed

\begin{prop}
\label{uyuy} Let $m=2$ and suppose $s\in R^*$. Then

(a) If $h$ is isotropic then
$$
d(s)=\frac{|A|^2-|\r|^2-|A^*||N||\m|}{|R^*|}.
$$

(b) If $h$ is non-isotropic then
$$
d(s)=\frac{|A|^2-|\r|^2}{|R^*|}.
$$
\end{prop}

\noindent{\it Proof.} The number of primitive vectors of $V$ is $|A|^2-|\r|^2$. Of these, $|A^*||N||\m|$ have length in $\m$
if $h$ is isotropic, and 0 if $h$ is non-isotropic, by Lemma \ref{tiso}. Suppose $s\in R^*$ and let $u\in V$ be primitive of length $h(u,u)=s$.
By Lemma \ref{orz2} there is $w\in V$ such that $u,w$ is an orthogonal basis of $V$. The stabilizer of $u$ in $U$ consists of all $g\in U$
such that $gu=u$ and $gw=zw$ for some $z\in N$. Thus
the number of $U$-conjugates of primitive vectors of unit length is always the same. Since there are $|R^*|$ possible units,
the result follows.\qed

\begin{lem}
\label{cam1} If $m=1$ then $|U|=|N|$.
\end{lem}

\noindent{\it Proof.} This is clear.\qed

\begin{prop}
\label{cam2} Let $m=2$. If $h$ is isotropic then
$$
|U|=\frac{|A^*||A||N|}{|R|},
$$
and if $h$ is not isotropic then
$$
|U|=\frac{(|A|^2-|\r|^2) |N|}{|R^*|}=\frac{|A^*|(|A|+|\r|) |N|}{|R^*|}.
$$
\end{prop}

\noindent{\it Proof.} Suppose first $h$ is isotropic. Then there is a primitive vector $u$ such that $h(u,u)=0$.
As in the proof of Theorem \ref{crux} we may find $v\in V$ such that the Gram matrix of $u,v$ is
$$\left(
  \begin{array}{cc}
    0 & 1 \\
    1 & 0 \\
  \end{array}
\right).
$$
Note that $u,v$ is a basis of $V$. A calculation shows that the stabilizer, say $S$, of $u$ consists of all $g\in U$ that relative to the basis $u,v$
have matrix
$$\left(
  \begin{array}{cc}
    1 & b \\
    0 & 1 \\
  \end{array}
\right),
$$
where $b+b^*=0$. Since the trace map $A\to R$ is surjective, the number of such $b$ is $|A|/|R|$.
On the other hand, by Lemma \ref{nuni}, we have $[U:S]=|A^*||N|$, as required.

The case when $h$ is non-isotropic follows from Proposition~\ref{uyuy} and Lemma \ref{cam1}.\qed

\begin{note}{\rm We see from above that $|U_2(A)|$ is strictly larger in the non-isotropic case.}
\end{note}

\begin{note}{\rm Take $m=2$ and choose $h$ to be isotropic. Computing $|U|$ by means of Proposition \ref{uyuy} and
Lemma \ref{cam1} yields
$$
|U|=\frac{|N||A^*|(|A|+|\r|-|N||\m|)}{|R^*|}.
$$
It follows from Proposition \ref{cam2} that
$$
\frac{|A|+|\r|-|N||\m|}{|R^*|}=\frac{|A|}{|R|}.
$$
We have verified this curious-looking identity independently of $|U|$. We omit the details
as they are not relevant to our study of $U$.}
\end{note}

\section{Stabilizer of a Primitive Vector of Non Unit Length}
\label{sec:stapri}

The goal of this section is to find the order of the stabilizer $S_u$ in $U$ of a primitive vector $u\in V$
of non-unit length $h(u,u)=r$. Such $u$ does not exist when $m=1$, while $|S_u|$ was already determined in \S\ref{sec:rank2} when $m=2$.
We therefore assume throughout this section that $m>2$. We first obtain some general information about $S_u$ and then
derive a formula for $|S_u|$.

As described in the proof of Theorem
\ref{thm:transitivity}, we can find
$v,w_1,\dots,w_{m-2}\in V$ in such a way that $u,
v,w_1,\dots,w_{m-2}$ is a basis of $V$, relative to which the
Gram matrix of $h$ is
$$
B = \begin{pmatrix}
    r & 1 & 0 & \cdots & 0 \\
    1 & 0&0 &\cdots & \vdots\\
    0 & 0 & 1 & \cdots & 0 \\
    \vdots & & & \ddots & \vdots \\
    0 & 0 & 0 & \cdots & \delta \\
\end{pmatrix},\quad \delta \in \{1,\epsilon\}.
$$
It will be convenient to let $D\in\GL_{m-2}(R)$ denote the diagonal bottom right corner of $B$. We also let $U'$ stand for the unitary group
associated to $W=\langle w_1,\dots,w_{m-2}\rangle$ and the restriction $h'$ of $h$ to $W$.
Note that $h'$ is of the same kind as $h$, as defined in \S\ref{sec:class}.

Given a $1\times (m-2)$ row vector $C$, an
$(m-2) \times (m-2) $ matrix $Y$,
a scalar $a$ and an $(m-2) \times 1$ column vector $X$, all with entries in $A$, we let
\[ G(C,Y,a,X) =
    \begin{pmatrix}
        1   & a & C \\
        0   & 1-ra  & -rC \\
        0   & X &   Y
    \end{pmatrix}.
    \]
Let $g\in S_u$. Observe that, with respect to the basis $u,v,w_1,\dots,w_{m-2}$,
the matrix of $g$  must be of the form $G(C,Y,a,X)$ for some
$C$, $Y$, $a$ and $X$. This follows from
$gu = u$,
$h(u,gv) = 1$ and $ h(u,gw_i) = 0$ for $i= 1,\dots,m-2$.

\begin{prop} \label{prop:CompletionExistence}
    Let $C\in A^{m-2}$. Then there exist $Y$, $a$ and $X$ so that $G(C,Y,a,X)$
    represents an element of $S_u$.
\end{prop}

\begin{proof}
    By definition $G=G(C,Y,a,X)$ represents an element of $U$ if and
    only if ${G^*}'BG = B$. Therefore
    it suffices to find $Y$, $a$ and $X$ so that
    \begin{equation}\label{eq:Yconstraint}
        {Y^*}'D Y = D+r{C^*}'C,
    \end{equation}
            \begin{equation}\label{eq:aX1}
                -ra^*a+(a+a^*) + {X^*}'D X = 0,
            \end{equation}
            \begin{equation}\label{eq:aX2}
                {Y^*}'D X = {C^*}'(ar-1).
            \end{equation}

   First consider (\ref{eq:Yconstraint}). We assert the existence of an upper triangular matrix $Y$, with unit diagonal entries, non-unit off-diagonal entries, and satisfying (\ref{eq:Yconstraint}). To see this,
    first note that $T=r{C^*}'C\in M_{m-2}(A)$ is hermitian. We next find $Y$ by successively computing its entries from top to bottom
    and left to right. The first column of (\ref{eq:Yconstraint}), written in row form, reads as follows:
    $$
    (Y_{11}^*,Y_{12}^*,Y_{13}^*,\dots,Y_{1,m-2}^*)Y_{11}=(1+r T_{11},rT_{12}^*,rT_{13}^*,\dots,rT_{1,m-2}^*).
    $$
    Since $r\in\m$ we can certainly find $Y_{11}\in A^*$ such that $Y_{11}^*Y_{11}=1+rT_{11}$ and then, using $Y_{11}\in A^*$,
    we can find $Y_{1i}\in\r$ such that $Y_{1i}^*Y_{11}=rT_{1i}^*$ for $i=2,\dots,m-2$.

    The second column of (\ref{eq:Yconstraint}), also written in row form, reads as follows:
    $$
    (Y_{11}^*,Y_{12}^*,Y_{13}^*,\dots,Y_{1,m-2}^*)Y_{12}+(0,Y_{22}^*,Y_{23}^*,\dots,Y_{2,m-2}^*)Y_{22}=
    (r T_{12},1+rT_{22},rT_{23}^*,\dots,rT_{2,m-2}^*).
    $$
    The equation $Y_{11}^*Y_{12}=rT_{12}$ is compatible with what we encountered before. Since $Y_{12}^*Y_{12}\in\m$
    we can find $Y_{22}\in A^*$ satisfying $Y_{12}^*Y_{12}+Y_{22}^*Y_{22}=1+rT_{22}$. Then, using $Y_{22}\in A^*$,
    we can find $Y_{2i}\in\r$ such that $Y_{1i}^*Y_{12}+Y_{2i}^*Y_{22}=r T_{2i}^*$ for $i=3,\dots,m-2$.

    Proceeding this way, it is easy to see that we can choose $Y$ as specified.
    Note that $\delta$ will be involved in these calculations
    only when solving for $Y_{m-2,m-2}$ in
    $$
    Y_{1,m-2}^*Y_{1,m-2}+\cdots+Y_{m-1,m-2}^*Y_{m-1,m-2}+\delta Y_{m-2,m-2}^*Y_{m-2,m-2}=\delta+rT_{m-2,m-2}.
    $$
This has a solution since $$1+\delta^{-1}(rT_{m-2,m-2}-Y_{1,m-2}^*Y_{1,m-2}-\cdots-Y_{m-1,m-2}^*Y_{m-1,m-2})\in 1+\m.$$

    Next consider (\ref{eq:aX1}) and (\ref{eq:aX2}). Substituting
    for $X$ in (\ref{eq:aX1}) using first (\ref{eq:aX2}) and then (\ref{eq:Yconstraint}) we obtain
    \begin{equation}\label{eq:aconstraint}
        ra^*a-(a+a^*) = (1-ar)^*C(D+r{C^*}'C)^{-1}{C^*}'(1-ar).
    \end{equation}
    Thus, it suffices to show that we can find $a$ that satisfies
    (\ref{eq:aconstraint}).
    Letting $$k = C(D+r{C^*}'C)^{-1}{C^*}'\in R,$$
    collecting like terms in (\ref{eq:aconstraint}), and using $1-rk\in R^*$,
    we can rewrite (\ref{eq:aconstraint}) as follows:
    \begin{equation}\label{eq:balanc}
        ra^*a - (a+a^*)= k(1-rk)^{-1}.
    \end{equation}
   We show in Lemma \ref{lem:quadraticequation}
    below that there is some $a$ that satisfies this
    last equation.
\end{proof}

\begin{lemma}\label{lem:quadraticequation}
    For any $r \in \m$ and
    $t \in R$, there is some $a\in A$ such that  \[ra^*a -(a+a^*) = t.\]
\end{lemma}
\begin{proof}
    Define $f: A \to R$ by $f(a) = ra^*a - (a+a^*) -t$. Observe that $$
    f(a+\frac12 f(a)) = \frac{1}{2}rf(a)(a+a^*+\frac12 f(a)).$$ Now $r$ is nilpotent,
    so for any sequence satisfying $a_{i+1} = a_i+\frac12 f(a_i)$,
    $f(a_i)=0$ for sufficiently large $i$.
\end{proof}

\begin{thm} \label{su}
Let $u\in V$ be a primitive vector of non unit length $h(u,u)$.
Let $S_u$ be the stabilizer of $u$ in $U$.
Suppose $m>2$ and let $U'$ be the unitary group associated to a hermitian form of the same kind as $h$ defined on a free $A$-module of rank $m-2$.
Then
$$
|S_u|=|U'|\times |A|^{m-1}/|R|.
$$
\end{thm}


\begin{proof} Let $C\in A^{m-2}$.  By Proposition \ref{prop:CompletionExistence} there is at least one $g=G(C,Y,a,X)\in S_u$ and we fix this element.
It follows from (\ref{eq:Yconstraint}) that $Y$ is invertible.
We claim that the $Y_1$'s such that $G(C,Y_1,a_1,X_1)\in S_u$
for some $a_1,X_1$ are precisely those of the form $Y_1=g'Y$ for some $g'\in U'$. Indeed, we may extend any $g'\in U'$ to $g_1=G(0,g',0,0)\in S_u$. Then $g_1g=G(C,g'Y,a,g'X)$ is in $S_u$. Suppose, conversely, that $G(C,Y_1,a_1,X_1)\in S_u$. Then by  (\ref{eq:Yconstraint})
$$
    {Y_{1}^*}'D Y_1 = D - r{C^*}'C = {Y^*}'D Y.
$$
As indicated above, $Y$ and $Y_1$ must be invertible, so
$${(Y_1Y^{-1})^*}'D Y_1Y^{-1} = D.$$
Therefore $Y_1Y^{-1}=g'\in U'$ and $Y_1=g'Y$, as claimed. Thus there are $|U'|$ choices for $Y_1$ such that $G(C,Y_1,a_1,X_1)\in S_u$.

Let $L_u$ be the group of all $G(0,I,\alpha,0)\in S_u$. We claim that $g_1=G(C,Y,a_1,X_1)\in S_u$ if and only if $g_1=gg_0$ with $g_0\in L_u$. Indeed, given $g_0\in L_u$ we have $gg_0=G(C,Y,a_1,X_1)$. Conversely, if $g_1=G(C,Y,a_1,X_1)\in S_u$ then
$g_0=g^{-1}g_1$ fixes $w_1,\dots,w_{m-2}$, so $g_0$ preserves $\langle u,v\rangle=\langle w_1,\dots,w_{m-2}\rangle^\perp$.
But $g_0$ also fixes $u$, so necessarily $g_0\in L_u$, with $g_1=gg_0$.

Let $H$ be the unitary group of $\langle u,v \rangle$ and let $H_u$ be the stabilizer of $u$ in $H$. Clearly
$L_u\cong H_u$. By Lemma \ref{nuni}
we know that $|H_u|$ is independent of $r$, as long as $r \in \m$. So,
to compute $|H_u|$, we can assume for the moment that $r = 0$. Thus
$|H_u| = |\{\alpha \in A: \alpha^* + \alpha = 0\}| = |A|/|R|$, since the trace map is surjective. The result now follows.
\end{proof}

\begin{note}{\rm The formula is still correct if $m=2$, as it reduces to $|S_u|=|A|/|R|$, as expected.}
\end{note}

\begin{note}{\rm A special case of this formula, namely for $O_m(q)$, appears in \cite{RW}, pg. 72.}
\end{note}

%
%
%
%


\section{Weil character degrees}

In this section $\ell\geq 1$ and $A=\mathcal{R}/\mathcal{P}^{2\ell}$, with $F/K$ ramified, and involution $*$ inherited from~$\mathcal{R}$.
We wish to determine the degrees of the irreducible constituents of the top layer of the Weil representation of $U_m(A)$, as described
in \cite{HS}. For this purpose, we fix the $*$-invariant ideal $\r^\ell$ of $A$ and consider $B=\overline{A}$, $\overline{V}$ and $\overline{h}$, as in \S\ref{sec:class}, recalling that $h$ and $\overline{h}$ are of the same kind. We also let $\overline{U}=U_m(B)$, as in \S\ref{sec:Epimorphism}.

Let $\ell = 2f$ or $\ell=2f-1$ depending on whether $\ell$ is even or odd. As mentioned in the Introduction, the degrees of the irreducible constituents of the top layer of the Weil character of $U_{m}(A)$ are equal to $[U_m(B):S_u]/2q^{\ell-f}$, where
$u$ is a primitive vector in $\overline{V}$.

The formulas in Corollary \ref{pom2} and Theorem \ref{su}, plus the orders of finite orthogonal groups (see \cite[Theorem 6.17]{NJ}), are all that is needed to compute $[U_m(B):S_u]$ in general. 
The easiest case of these calculations, namely when $\ell=1$ and $B\cong F_q$, were performed in \cite{HS}.
We will make no use of this case. The reader can easily verify that our general formulas reduce to those of \cite{HS} when $\ell=1$.

We have an involution $\circ$ on $B$ inherited from $A$ and we let $\overline{R}$ stand for the fixed ring of $\circ$. As indicated at the end of \S\ref{sec:stapri}, we have $|\overline{R}|=q^f$.
The maximal ideal of $\overline{R}$ will be denoted by $\overline{\m}$ and we further set $\overline{N}=\{z\in B\,|\, zz^\circ=1\}$.

When $m=1$, $U_1(A) \le A^*$ is abelian so the degrees of the irreducible constituents of the Weil character will be $1$.
Thus we assume throughout this section that $m\geq 2$. We begin by computing $[U_m(B):S_u]$ when $m=2$, as this case
has essentially been done in Propositions \ref{nuni} and \ref{uyuy}. The case $m=2$ is reobtained later as part of
the general case $m\geq 2$.

\begin{prop} Suppose $m=2$.

(a) If $\overline{h}$ is non-isotropic, then $t$ must be a unit and $[U_2(B):S_u]= q^{2\ell-f-1}(q+1)$.

(b) If $\overline{h}$ is isotropic, then $[U_2(B):S_u]= q^{2\ell-f-1}(q-1)$ if $t$ is a unit and $[U_2(B):S_u]= 2q^{2\ell-f-1}(q-1)$ if $t$ is not a unit.
\end{prop}

\begin{proof} (a) We know that $t$ must be a unit by Lemma \ref{tiso}. Moreover, Proposition \ref{uyuy} gives
$$[U_2(B):S_u] = \frac{|B|^2-|\overline{\r}|^2}{|\overline{R}^*|} = \frac{(q^{2\ell}-q^{2\ell-2})}{(q^f-q^{f-1})} = q^{2\ell-f-1}(q+1).$$

(b) If $t$ is not a unit, then by Proposition \ref{nuni}, $[U_2(B):S_u]=|B^*||\overline{N}|$. The norm map is a group homomorphism from $B^*$, which has order $q^{\ell}-q^{\ell-1}$ onto $\overline{R}^{*2}$, which has order $\frac{(q^{f}-q^{f-1})}{2}$. Therefore, $[U_2(B):S_u]= 2q^{\ell-f}(q^{\ell}-q^{\ell-1}) = q^{2\ell-f-1}(q-1)$.

If $t$ is a unit, then by Proposition \ref{uyuy}, $$\begin{array}{rl}
[U_2(B):S_u] &=\dfrac{|B|^2-|\overline{\r}|^2-|B^*||\overline{N}||\overline{\m}|}{|\overline{R}^*|}
=\frac{(q^{2\ell}-q^{2\ell-2})-(q^{\ell}-q^{\ell-1})(2q^{\ell-f})(q^{f-1}))}{(q^f-q^{f-1})} \\
&= \frac{(q^{2\ell-2}(q^2-1)-(2q^{2\ell-2})(q-1)}{q^{f-1}(q-1)} =
\frac{q^{2\ell-2}(q+1)-2q^{2\ell-2}}{q^{f-1}} = q^{2\ell-f-1}(q-1).
\end{array}$$

\end{proof}

We know proceed to the general case $m\geq 2$. 
Suppose first $t$ is a unit. Then $S_u \simeq U_{m-1}(B)$ is a unitary group for a hermitian form $b$, namely the restriction
of $\overline{h}$ to $\langle u\rangle^\perp$. If $m$ is even then $|U_{m-1}(B)|$ is independent of $b$.
When $m$ is odd this is no longer true and in this case we will state our result according on the kind of $b$. But this depends on the kind of $\overline{h}$, which is the same as that of $h$, as well as on $t$ and~$q$.
Indeed, we easily see that if $m$ is odd then $b$ is of kind $I$ if and only if one of the following holds:

a) $\bar{h}$ is of kind $I$, $-1$ is a square, and $t$ is a square,

b) $\bar{h}$ is of kind $I$, $-1$ is not a square, and $t$ is not a square,

c) $\bar{h}$ is of kind $II$, $-1$ is not a square, and $t$ is a square, and

d) $\bar{h}$ is of kind $II$, $-1$ is a square, and $t$ is not a square.

Otherwise, $b$ is of kind $II$.

\begin{thm} Suppose $t$ is a unit.

(i) If $m=2r+1$ is odd, then
$$[U_m(B):S_u] = \begin{cases} q^{m\ell-m+r-f+1}(q^r+1) \text{ if $b$ is of kind $I$, and } \\
                             q^{m\ell-m+r-f+1}(q^r-1) \text{ if $b$ is of kind $II$.} \end{cases} $$


(ii) If $m=2r$ is even, then
$$[U_m(B):S_u] = \begin{cases} q^{m\ell-m+r-f}(q^r-1) \text{ if $\bar{h}$ is of kind $I$, and } \\
                             q^{m\ell-m+r-f}(q^r+1) \text{ if $\bar{h}$ is of kind $II$.} \end{cases} $$

\end{thm}

\begin{proof} {\it (i)} By Corollary \ref{pom2}, if $\ell=2f-1$ then
$$\begin{array}{rcl}
[U_m(B):S_u] & = & \dfrac{|U_m(B)|}{|U_{m-1}(B)|} \\
             & & \\
             & = & \dfrac{q^{m^2(\ell-1)/2}|O_m(q)|}{q^{(m-1)^2(\ell-1))/2}|O_{m-1}(q)|} \\
             & & \\
             &=&  q^{m\ell - m - f + 1}\dfrac{|O_m(q)|}{|O_{m-1}(q)|},
             \end{array}$$
while if $\ell=2f$ then
$$\begin{array}{rcl}
[U_m(B):S_u] & = & \dfrac{|U_m(B)|}{|U_{m-1}(B)|} \\
             & & \\
             & = & \dfrac{q^{(m^2(\ell-1)+m)/2}|O_m(q)|}{q^{((m-1)^2(\ell-1)+(m-1))/2}|O_{m-1}(q)|} \\
             & & \\
             &=&  q^{m\ell - m -  f + 1} \dfrac{|O_m(q)|}{|O_{m-1}(q)|}.
             \end{array}$$
The result now follows, since when $m=2r+1$, \cite[Theorem 6.17]{NJ} gives
$$\dfrac{|O_m(q)|}{|O_{m-1}(q)|} = \begin{cases} q^r(q^r+1) \text{ if $b$ is of kind $I$, and } \\
                                 q^r(q^r-1) \text{ if $b$ is of kind $II$.} \end{cases} $$

\noindent {\it (ii)} When $m=2r$, \cite[Theorem 6.17]{NJ} gives
$$\dfrac{|O_m(q)|}{|O_{m-1}(q)|} = \begin{cases} q^{r-1}(q^r-1) \text{ if $\bar{h}$ is of kind $I$, and } \\
                                 q^{r-1}(q^r+1) \text{ if $\bar{h}$ is of kind $II$.} \end{cases} $$
The remaining calculations are exactly as above. This completes the proof.
\end{proof}

We proceed to the case $t \in \m$. Note that when $m=2$ this case can only occur when $\overline{h}$ is of kind I.

\begin{thm} Suppose $t \in \m$.

(i) If $\ell=2f-1$ is odd and $m$ is odd, then
$$[U_m(B):S_u]=q^{m\ell-m-\ell+f}(q^{m-1}-1).$$

(ii) If $\ell=2f-1$ is odd and $m=2r$ is even, then
$$[U_m(B):S_u]=\begin{cases} q^{m \ell -m - \ell+f}(q^r-1)(q^{r-1}+1) \text{ if $\bar{h}$ is of kind $I$,}\\
                             q^{m \ell -m - \ell+f}(q^r+1)(q^{r-1}-1) \text{ if $\bar{h}$ is of kind $II$.}
               \end{cases}$$

(iii) If $\ell=2f$ is even and $m$ is odd, then
$$[U_m(B):S_u]=q^{\ell m - m-\ell + f +1 }(q^{m-1}-1).$$

(iv) If $\ell=2f$ is even and $m=2r$ is even, then
$$[U_m(B):S_u]=\begin{cases} q^{\ell m - m- \ell + f +1} (q^r-1)(q^{r-1}+1) \text{ if $\bar{h}$ is of kind $I$,}\\  q^{\ell m - m-\ell + f +1} (q^r+1)(q^{r-1}-1) \text{ if $\bar{h}$ is of kind $II$.} \end{cases} $$
\end{thm}

\begin{proof} Suppose $\ell=2f-1$ is odd. Using Theorem \ref{su} and Corollary \ref{pom2},
$$\begin{array}{rcl}
[U_m(B):S_u] & = & \dfrac{|U_m(B)|}{|S_u|} = \dfrac{|U_m(B)||\overline{R}|}{|U_{m-2}(B)||B|^{m-1}}  \\
             &  & \\
             & = & \dfrac{q^{m^2(\ell-1)/2}|O_m(q)|q^f}{q^{(m-2)^2(\ell-1))/2}|O_{m-2}(q)|q^{\ell(m-1)}} \\
             & & \\
             & = & q^{\ell m - 2m - \ell + 2 + f } \dfrac{|O_m(q)|}{|O_{m-2}(q)|}.
             \end{array}$$
On the other hand, if $\ell=2f$ is even,
$$\begin{array}{rcl}
[U_m(B):S_u] & = & \dfrac{|U_m(B)|}{|S_u|} = \dfrac{|U_m(B)||\overline{R}|}{|U_{m-2}(B)||B|^{m-1}}  \\
             &  & \\
             & = & \dfrac{q^{(m^2(\ell-1)+m)/2}|O_m(q)|q^f}{q^{((m-2)^2(\ell-1)+(m-2))/2}|O_{m-2}(q)|q^{\ell(m-1)}} \\
             & & \\
             & = & q^{\ell m - 2m - \ell+ 3+ f } \dfrac{|O_m(q)|}{|O_{m-2}(q)|}.
             \end{array}$$
Using the orders of finite unitary groups from \cite[Theorem 6.17]{NJ}, we have that
when $m$ is odd,
$$\dfrac{|O_m(q)|}{|O_{m-2}(q)|}=q^{m-2}(q^{m-1}-1),$$
and when $m=2r$ is even,
$$\dfrac{|O_m(q)|}{|O_{m-2}(q)|}=\begin{cases} q^{m-2}(q^r-1)(q^{r-1}+1) \text{ if $\bar{h}$ is of kind $I$,}\\
                                               q^{m-2}(q^r+1)(q^{r-1}-1) \text{ if $\bar{h}$ is of kind $II$.}
                                               \end{cases}$$
It is now straightforward to complete the calculation and obtain the formula in each of the four cases.
\end{proof}

\end{document}